\documentclass{amsart}

\usepackage{amsmath,amsfonts,amssymb,amsthm,amsrefs,graphicx}
\usepackage[margin=1in]{geometry}

\usepackage{tikz}
\usetikzlibrary{shapes.geometric}
\tikzset{
    buffer/.style={
        draw,
        regular polygon,
        regular polygon sides=3,
        fill=red,
        node distance=1cm,
        minimum height=3em
    }
}

\newcommand{\norm}[1]{\left\lVert {#1} \right\rVert}

\newcommand{\C}{{\mathbb{C}}}

\newcommand{\N}{{\mathbb{N}}}

\newcommand{\sA}{{\mathcal{A}}}
\newcommand{\sB}{{\mathcal{B}}}

\newcommand{\sP}{{\mathcal{P}}}

\newtheorem{theorem}{Theorem}
\newtheorem{conjecture}{Conjecture}
\newtheorem{prop}{Proposition}

\newtheorem{lemma}{Lemma}
\newtheorem{question}{Question}
\newtheorem*{theorem*}{Theorem}

\theoremstyle{definition}
\newtheorem{definition}{Definition}

\theoremstyle{remark}
\newtheorem{remark}{Remark}

\author{Jennifer Brooks}
\address{Department of Mathematics, Brigham Young University, Provo, UT, 84602}
\email{jbrooks@mathematics.byu.edu}
\author{Dusty Grundmeier}
\address{Department of Mathematics, Harvard University, Cambridge, MA, 02138}
\email{deg@math.harvard.edu}

\title{Sum of Squares Conjecture: the Monomial Case in $\mathbb{C}^3$}

\begin{document}

\maketitle

\begin{abstract}
The goal of this article is to prove the Sum of Squares Conjecture for real polynomials $r(z,\bar{z})$ on $\mathbb{C}^3$ with diagonal coefficient matrix. This conjecture describes the possible values for the rank of $r(z,\bar{z}) \norm{z}^2$ under the hypothesis that $r(z,\bar{z})\norm{z}^2=\norm{h(z)}^2$ for  some holomorphic polynomial mapping $h$.  Our approach is to connect this problem to the degree estimates problem for proper holomorphic monomial mappings from the unit ball in $\mathbb{C}^2$ to the unit ball in $\mathbb{C}^k$.  
D'Angelo, Kos, and Riehl proved the sharp degree estimates theorem in this setting, and we give a new proof using techniques from commutative algebra.  We then complete the proof of the Sum of Squares Conjecture in this case using similar algebraic techniques.
\end{abstract}

\section{\label{sec:intro} Introduction}

Let $r(z, \bar{z})$ be a real polynomial on the diagonal of $\mathbb{C}^n \times \mathbb{C}^n$ and suppose there is a holomorphic polynomial mapping $h$ on $\mathbb{C}^n$ such that $r(z,\bar{z}) \norm{z}^2 = \norm{h(z)}^2$.  The rank of a real polynomial is the rank of its matrix of coefficients. The goal of this paper is to study the possible ranks $\rho$ of the real polynomial $\norm{h}^2$. In particular, we study the following conjecture of Ebenfelt (see \cite{E:17}) and its connections with other classical problems in several complex variables.

\begin{conjecture}[Sum of Squares (SOS) Conjecture \cite{E:17}] \label{thm:sos}
Suppose $n \geq 2$, and define
\begin{equation}\label{eq:k0}
k_0 = \max\left\{k \in \N_0: \frac{k(k+1)}{2} < n-1\right\}.
\end{equation}
Let $r(z,\bar{z})$ be a real polynomial on the diagonal of $\mathbb{C}^n \times \mathbb{C}^n$, and suppose that $r(z,\bar{z}) \norm{z}^2$ is a squared norm, i.e.,
\begin{equation*}
r(z,\bar{z})\norm{z}^2=\norm{h(z)}^2.
\end{equation*}
Let $\rho$ be the rank of $\norm{h}^2$. Then either
\begin{equation*}
\rho \geq (k_0 +1)n -\frac{k_0(k_0+1)}{2},
\end{equation*}
 or there exists an integer $0\leq k \leq k_0 <n$ such that
\begin{equation*}
n k -\frac{k(k-1)}{2} \leq \rho \leq n k.
\end{equation*}
\end{conjecture}

In other words, not all natural numbers are possible ranks for $\norm{h}^2$. For instance, Huang's Lemma \cite{H1} gives either $\rho =0$ or $\rho \geq n$, and this proves the conjecture when $n=2$.

In \cite{GB}, we proved Conjecture \ref{thm:sos} if $r$ is itself positive definite using techniques from commutative algebra, in particular, Macaulay's estimate on the growth of a homogeneous ideal in a polynomial ring. In this paper, we prove the conjecture in the diagonal case in three variables.

\begin{theorem}\label{thm:SOS for diagonal r in 3 var}
Conjecture \ref{thm:sos} holds if $n=3$ and the coefficient matrix of $r$ is diagonal.
\end{theorem}

It is not hard to show that it suffices to prove Conjecture 1 in the bihomogeneous case, and so for the remainder of the paper, we assume that $r$ is bihomogeneous of bi-degree $(d-1,d-1)$.

We will prove Theorem \ref{thm:SOS for diagonal r in 3 var} by relating it to another natural rank question.  Whereas Conjecture \ref{thm:sos} simply asks for the possible ranks of $r(z,\bar{z}) \norm{z}^2$ in the case in which it is the squared norm of a holomorphic polynomial mapping, one expects in general that the rank depends on the bi-degree $(d-1,d-1)$. Thus one might also ask the following:

\begin{question}
Let $r(z, \bar{z})$ be a real bihomogeneous polynomial on the diagonal of $\mathbb{C}^n \times \mathbb{C}^n$ of bi-degree $(d-1,d-1)$.  Let $\rho$ be the rank of the polynomial $r(z,\bar{z}) \norm{z}^2$.  Can we obtain a sharp lower bound on $\rho$ in terms of $n$ and $d$?
\end{question}

Without some additional hypotheses on $r$, the question is uninteresting because one can easily construct examples with a low rank but arbitrary degree.  With some mild hypotheses on $r$ to rule out these sorts of situations, one can obtain some interesting results.  Lebl and Peters \cite{LP:11,LP:12} have considered this question for $r$ having a diagonal coefficient matrix and for arbitrary $n$.  They show that when $n=3$, if the terms of $r$ have no common monomial factor of positive degree and if $r$ satisfies some ``connectedness" hypothesis to be made precise later, then
$$\rho \geq \frac{d+5}{2}. $$
They use this rank estimate to obtain a new proof of the sharp degree bounds for proper, monomial maps from the unit ball $B_2$ in $\mathbb{C}^2$ to the unit ball $B_k$ in $\mathbb{C}^k$. The original proof of the sharp degree estimate for proper holomorphic monomial maps from $B_2$ to $B_k$ is due to D'Angelo, Kos, and Riehl \cite{DKR:03}. 

Our approach in this paper is to use tools from commutative algebra to study these two related rank problems.  In particular, we consider several homogeneous ideals naturally associated with the bihomogeneous polynomial $r$.  Our rank estimates can be made in terms of the Hilbert functions and graded Betti numbers for these ideals. We give a new proof of the result of Lebl and Peters using this language.  We then use this theorem to obtain the desired lower bound on $\rho$ needed for Theorem \ref{thm:SOS for diagonal r in 3 var} for polynomials of high degree. For polynomials of low degree, the signature pair $(P,N)$ of $r$ plays a crucial role.  We treat this case in two steps.  First, we prove several general results about the rank of $r(z,\bar{z}) \norm{z}^2$ when the number $N$ is small. We are then left with a small number of remaining cases that can be easily analyzed.  We end the paper by discussing how our algebraic proof of the result of Lebl and Peters leads to a new proof of the sharp degree estimates theorem of D'Angelo, Kos, and Riehl for proper, monomial mappings from $B_2$ to $B_k$.

These results are part of a long-standing program in several complex variables to classify proper rational mappings from the unit ball in $\mathbb{C}^n$ to the unit ball in $\mathbb{C}^k$. See for instance \cite{DL:09, JPD(scv):93, JPD:16, FF:survey,DX, H1, E1, E:17} and their references. When the codimension is small, there is remarkable additional structure. For instance, when $n>2$ and $k<2n-1$, Faran \cite{JF:86} proved all proper holomorphic mappings, that extend smoothly to the boundary, from $B_n$ to $B_k$ are spherically equivalent to $z\mapsto (z,0)$. The intervals of codimensions where no new maps appear is called a gap. The Huang-Ji-Yin Gap Conjecture, stated in \cite{HJY:09}, completely classifies these gaps. Many researchers have made significant contributions to this line of research.  See especially \cite{H1,HJX,HJY1,E:17,E1}.  We will not discuss the Gap Conjecture in detail here, but refer the interested reader to \cite{E:17} for a concise overview and relevant references. For the purposes of this article, the relevant fact is that there is a connection between the Gap Conjecture and Hermitian sums of squares. More precisely, Ebenfelt \cite{E:17,E1} proved that his Sum of Squares Conjecture implies the Gap Conjecture.

\section{\label{sec:alg} The Commutative Algebra Framework}

If $r(z, \bar{z})$ is a real bihomogeneous polynomial on the diagonal of $\mathbb{C}^n \times \mathbb{C}^n$ of bi-degree $(d-1,d-1)$, it has a {\it holomorphic decomposition}
$$r(z, \bar{z}) = \norm{f(z)}^2 - \norm{g(z)}^2 $$
for homogeneous holomorphic polynomial mappings $f=(f_1,\ldots,f_P)$ and $g=(g_1,\ldots,g_N)$ of degree $d-1$. (See \cite{JPD:Carus} for an extensive discussion of this idea and numerous applications.) Although in general the holomorphic decomposition of a real polynomial is not unique, if we require the $f_j$ and $g_k$ to be linearly independent, the pair $(P,N)$ is uniquely determined and is called the signature pair of $r$.  If we write $r(z,\bar{z}) = \sum c_{a b} z^a \bar{z}^b$, then $(P,N)$ is the signature pair of the coefficient matrix $(c_{ab})$ and $P+N$ is its rank.

Let $R=\mathbb{C}[z_1,\ldots,z_n]$.  $R$ is a graded ring, graded by degree.
We consider three homogeneous ideals in $R$ naturally associated with $r(z,\bar{z})$.  Let $I_f = \langle f_1,\ldots,f_P \rangle$, $I_g=\langle g_1,\ldots,g_N \rangle$, and $I_{f \oplus g}=\langle f_1,\ldots,f_P,g_1,\ldots,g_N \rangle$.
For any finitely-generated graded $R$-module $M$, we may write $M = \oplus_\ell M_\ell$, where $M_\ell$ is the component of $M$ in degree $\ell$.  The component $M_\ell$ is a vector space over $\mathbb{C}$, and we define
\begin{equation*}
H_M(\ell)=\dim_{\mathbb{C}} M_\ell.
\end{equation*}
$H_M$ is the {\it Hilbert function} of $M$.

The condition that $r(z,\bar{z}) \norm{z}^2$ is a squared norm implies a relationship between the Hilbert functions $H_{I_f}$ and $H_{I_{f \oplus g}}$.
\begin{lemma}\label{lemma: containment of ideals for Pn1}
If $r(z,\bar{z})$ is a bihomogeneous polynomial of bi-degree $(d-1,d-1)$ and if $r(z,\bar{z}) \norm{z}^2 = \norm{h(z)}^2$, then $(I_f)_d = (I_{f \oplus g})_d$, and thus
\begin{equation*}\label{eq:equality of H functions}
H_{I_f} (d) = H_{I_{f \oplus g}} (d).
\end{equation*}
\end{lemma}
The lemma was originally proved by D'Angelo \cite{JPD:05}, though it appeared in a slightly different form. See \cite{BG:19} for a proof of the lemma in this form.

We will make use of Macaulay's estimate on the growth of a homogeneous ideal $I$ in the polynomial ring $R$.  Macaulay's result gives an upper bound for the Hilbert function for $R/I$ and hence a lower bound for the Hilbert function for $I$. Before we can state the estimate, we need some notation.
\begin{definition}
Let $c$ and $\nu$ be positive integers.  The $\nu$-th Macaulay representation of $c$ is the unique way of writing
\begin{equation*}
c= \binom{k_\nu}{\nu} +\binom{k_{\nu-1}} {\nu-1} + \ldots + \binom{k_J} {J}
\end{equation*}
where $k_\nu > k_{\nu-1}> \ldots > k_J \geq J > 0$. We also write
\begin{equation*}
c^{<\nu>} = \binom{k_\nu + 1} {\nu+1} + \binom{k_{\nu-1} + 1}{ \nu } + \ldots + \binom{k_J + 1} {J + 1}.
\end{equation*}
\end{definition}
See \cite{Green:gin} and \cite{GreenRestrictions} for a more extensive discussion of these ideas, including proofs of the uniqueness of the $\nu$-th Macaulay representation of a positive integer and of the elementary properties of the function $c \mapsto c^{<\nu>}$. Moreover, we refer the reader to \cite{BG:19, G1, GB, GL, GLV} for similar approaches to problems in several complex variables using commutative algebra and algebraic geometry.

\begin{theorem}[Macaulay's estimate on the growth of an ideal \cite{Mac1}] \label{thm: Macaulay}  Let $I$ be an ideal in $R$ whose generators are homogeneous polynomials (not necessarily all of the same degree). Then
\begin{equation*}\label{Macaulay's estimate}
H _{R/I}(\ell+1) \leq H_{R/I}(\ell)^{<\ell>}.
\end{equation*}
\end{theorem}
Elementary linear algebra shows that
\begin{equation}\label{eq:linear alg relation between H functions}
H_{R/I}(\ell)+H_{I} (\ell)={\ell+n-1 \choose \ell}.
\end{equation}

Although the Hilbert function of a graded $R$-module is a useful invariant, the graded Betti numbers provide more detailed information.  We briefly recall the notions of minimal free resolutions and graded Betti numbers. See \cite{BG:19} for a more detailed discussion and additional applications to the study of Hermitian polynomials and Hermitian sums of squares.   A finitely-generated $R$-module $M$ need not be free.  Thus given a set $\{h_j :1 \leq j \leq J\}$ of homogeneous generators for $M$, there may be non-trivial relationships among the generators, i.e., elements $s_j \in R$ for which $\sum  s_jh_j=0$.  The set of all $ s=(s_1,\ldots,s_J)$ with this property forms a module of {\it syzygies} of $M$.

We can gain considerable information about $M$ through a free resolution.
To begin the construction, define a map
$\varphi_0$ from a graded free module $F_0$ to $M$ by sending the
generator $\varepsilon^0_j$ of the $j$-th summand of $F_0$ to the $j$-th homogeneous generator $h_j$ of $M$. Because
we want such a map of graded modules to be degree-preserving, we
regard this generator $\varepsilon^0_j$ as having the same
degree $d_j$ as $h_j$.  We use the notation $R(-d)$ to denote the shift of $R$ by $d$ so that $R(-d)_k=R_{k-d}$. We then write $F_0 = \oplus_j R(-d_j)$. Let $M_1=\ker \varphi_0 $.  $M_1$ is a finitely-generated graded $R$-module and is our first syzygy module.  Choose a finite set of homogeneous generators for $M_1$ and, as above, construct a map
$\varphi_1$ from a graded free module $F_1$ to $F_0$ with image
$M_1$.  Continuing in this manner, we obtain a graded free
resolution of $M=\operatorname{coker} \varphi_1$:
\begin{equation}\label{eq: free resolution}
\ldots \longrightarrow F_i \xrightarrow{\varphi_i} F_{i-1} \longrightarrow \ldots \longrightarrow F_1
\xrightarrow{\varphi_1} F_0.
\end{equation}
This sequence is exact, i.e.,  for each $i$,
$\operatorname{im}\varphi_{i+1} = \ker \varphi_i$.

The Hilbert Syzygy Theorem says that there is a
free resolution of length at most $n$, i.e.,
\begin{equation*}
0\longrightarrow F_k \xrightarrow{\varphi_k} F_{k-1} \longrightarrow \ldots \longrightarrow F_1 \xrightarrow{\varphi_1} F_0
\end{equation*}
with $k \leq n$ (the number of variables in $R$).
A  free resolution of an $R$-module is not
unique.  The {\it minimal free resolution}, however, will be
unique up to isomorphism.  The free resolution \eqref{eq: free
resolution} is minimal if and only if for each $i$,
$\varphi_i$ takes a basis of $F_i$ to a minimal set of generators
for $\varphi_i(F_i)=M_i$. In the minimal free resolution of a finitely-generated graded $R$-module $M$, the number $\beta_{i,j}$ of generators of $F_i$ in degree $j$ is uniquely determined.  These numbers $\beta_{i,j}$ are the {\it graded Betti numbers} of $M$.

We will make use of the following proposition relating Betti numbers and Hilbert functions:
\begin{prop}\label{prop: rel between Betti numbers and Hilbert}
Let $B_j=\sum_{i} (-1)^i \beta_{i,j}$.  Then
\begin{equation*}
H_M(\ell)=\sum_{j=0}^{\ell} B_j \binom{\ell-j+n-1}{n-1}.
\end{equation*}
\end{prop}

In the sections that follow, we will consider the graded Betti numbers for $I_f$, $I_g$, and $I_{f \oplus g}$ when the components of the mappings are monomials. It will thus be important to be able to find generating sets for the first syzygy module of a monomial ideal.
Let
$\{\, \varepsilon(a): |a| = d-1\,\}$ be the set of generators for
the free module $F_0$ in the minimal free resolution of $I_{f \oplus g}$,  i.e.,
$\varphi_0(\varepsilon(a))=z^a$. Set
$$\sigma(a,b)=\frac{z^b}{\gcd(z^a,z^b)}\varepsilon(a) - \frac{z^a}{\gcd(z^a,z^b)}\varepsilon(b).$$
Then $\{\sigma(a,b)\}$ is a set of generators
for $\ker \varphi_0 = M_1$. Such elements
$\sigma(a,b)$ of $M_1$ are called {\it divided
Koszul relations}.  Of course these relations do not in general
give a {\it minimal} set of generators for $\ker \varphi_0$.

\section{\label{sec:small numbers of negs} The SOS Conjecture for Small Numbers of Negatives}

The idea now is to estimate the rank $\rho$ of $\norm{h}^2$ in terms of the Hilbert functions  of the various ideals. If $I$ is a homogeneous ideal all of whose $k$ generators are of degree $d-1$, Macaulay's estimate and \eqref{eq:linear alg relation between H functions} yield
$$H_{R/I}(d) \leq \left( \binom{d-1+n-1}{d-1} - k\right)^{<d-1>}. $$
If we apply \eqref{eq:linear alg relation between H functions} again, we obtain a lower bound for $H_I(d)$:
\begin{equation}\label{eq:lower bound on H function}
H_I(d) \geq \binom{d+n-eq:1}{d} - \left( \binom{d-1+n-1}{d-1} - k\right)^{<d-1>}.
\end{equation}
When the degree $d-1$ and the number of variables $n$ are fixed, this lower bound depends only on the number of generators $k$ for $I$.  To simplify our notation moving forward, we make a definition.
\begin{definition}
Suppose $I$ is a homogeneous ideal with $k$ generators of degree $d-1$ in the ring $R=\mathbb{C}[z_1,\ldots,z_n]$.  Let $M_{n,d-1} \colon \mathbb{N}_0 \to \mathbb{N}_0$ be defined by
\begin{equation*}
M_{n,d-1}(k) = \binom{d+n-1}{d} - \left( \binom{d-1+n-1}{d-1} - k\right)^{<d-1>}.
\end{equation*}
We call $M_{n,d-1}$ the {\bf Macaulay function}. When no confusion can arise, we suppress the subscripts.
\end{definition}
The next lemma collects a number of useful statements about the Macaulay function. The proof of (a) appears in \cite{GB}.  Proofs of the remaining parts are similar and we omit them.
\begin{lemma}\label{lemma:properties of Macaulay function}
\begin{enumerate}

\item[(a)] If $0 \leq k \leq n$,
\begin{equation*}\label{eq:Macaulay fctn for at most n generators}
M(k) = nk-\frac{k(k-1)}{2}.
\end{equation*}

\item[(b)] If $0 \leq j \leq n$,
\begin{equation*}\label{eq:another way to write Macaulay fctn for a small number of generators}
M(n-j)=\frac{n(n+1)}{2} - \frac{j(j+1)}{2}.
\end{equation*}

\item[(c)] If $j \leq n-1$,
\begin{equation*}
M(n+j) = \frac{n(n+1)}{2} + nj-\frac{j(j+1)}{2} \label{eq:M(n+j) for n>j}
\end{equation*}
\end{enumerate}
\end{lemma}

This lemma sheds some light on the form of the inequalities appearing in the Sum of Squares Conjecture.  Using the notation of the Macaulay function, we may restate the conjecture:
\begin{conjecture}[SOS Conjecture, Version 2] \label{thm:sos RESTATED}
Suppose $n \geq 2$ and define $k_0$ as above in \eqref{eq:k0}.
Let $r(z,\bar{z})$ be a real bihomogeneous polynomial and suppose
\begin{equation*}
r(z,\bar{z})\norm{z}^2=\norm{h(z)}^2.
\end{equation*}
Let $\rho$ be the rank of $\norm{h}^2$. Then either
\begin{equation*}
\rho \geq M(k_0+1),
\end{equation*}
 or there exists an integer $k$ with $0\leq k \leq k_0 <n$ such that
\begin{equation*}
M(k) \leq \rho \leq n k.
\end{equation*}
\end{conjecture}

We first state and prove the result of \cite{GB} for completeness.
\begin{theorem}
The SOS Conjecture holds if $r$ has signature pair $(P,0)$.
\end{theorem}
\begin{proof}
In this case, $\rho = H_{I_f}(d)$.  We have the trivial upper bound
\begin{equation*}
H_{I_f}(d) \leq nP
\end{equation*}
because $(I_f)_d$ is spanned by $\{z_j f_k:1\leq j \leq n, 1\leq k \leq P\}$.

We use Macaulay's estimate to obtain the lower bound
$$H_{I_f}(d) \geq M(P). $$
Thus
$$M(P) \leq \rho \leq nP. $$
We conclude that, if $P \geq k_0+1$, $\rho \geq M(k_0 +1)$, whereas if $0 \leq P \leq k_0$, $\rho$ is in the interval of allowable ranks defined by
$$M(P) \leq \rho \leq nP. $$
\end{proof}

\begin{remark}
The ``gaps" in possible rank for $\norm{h}^2$ seem only to occur in the positive-definite case; in all of the subsequent results in which we prove that the SOS Conjecture holds for certain $r$ of signature pair $(P,N)$ for some $N\geq 1$, we do so by proving that the rank of $r(z,\bar{z}) \norm{z}^2$ exceeds $M(k_0+1)$.  Ebenfelt \cite{E:17} conjectures that this pattern is true in general, and we concur.
\end{remark}

Our next goal is to prove a number of results that show that the SOS Conjecture holds when $r$ has diagonal coefficient matrix with signature pair $(P,N)$ for $N \geq 1$, but small.
For all of the results in the remainder of this section, we assume $r(z,\bar{z})$ is a real bihomogeneous polynomial of bi-degree $(d-1,d-1)$ such that $r(z,\bar{z})\norm{z}^2=\norm{h(z)}^2$, and we always use $\rho$ to denote the rank of $\norm{h}^2$.  A number of these results hold under more general hypotheses and will be discussed in this greater generality in a forthcoming paper.

We need the following easy lemma.
\begin{lemma}[Halfpap and Lebl \cite{HL:sig}]\label{lemma:P at least n}
If $r$ has diagonal coefficient matrix with signature pair $(P,N)$ for $N \geq 1$, then $P\geq n$.
\end{lemma}

\begin{prop}\label{prop:N=1}
Suppose $r$ has diagonal coefficient matrix with signature pair $(P,1)$.  Then $\rho \geq M(k_0+1)$and hence the SOS Conjecture holds in this case.
\end{prop}
\begin{proof}
We have
$$\rho \geq H_{I_{f \oplus g}(d)}-H_{I_g}(d) \geq M(P+1) - n. $$
By Lemma \ref{lemma:P at least n},  $P\geq n$.  Thus by part (c) of Lemma \ref{lemma:properties of Macaulay function},
$$M(P+1) \geq M(n+1)=\frac{n(n+1)}{2}+n-1.$$
Because  $k_0 + 1 < n$ and $M$ is a strictly increasing function,
\begin{eqnarray*}
\rho &\geq& M(n+1)-n\\
&=& \frac{n(n+1)}{2} - 1\\
&=&M(n) - 1\\
&\geq&M(k_0 +1).
\end{eqnarray*}
\end{proof}

For $n=3$, the type of argument used in the proof of  Proposition \ref{prop:N=1} does not immediately generalize to show  that the SOS Conjecture holds for $r$ of signature pair $(P,2)$.  Indeed, if we use $nN=6$ as the upper bound for $H_{I_g}(d)$ and $M(P+2)$ as the lower bound for $H_{I_{f \oplus g}}(d)$, we obtain the estimate
$$\rho \geq M(P+2)-6 \geq M(5)-6=9-6=3. $$
This inequality is true, but the lower bound is less than the hoped-for lower bound of $M(k_0+1)=M(2)=5$.  A more subtle argument is required, and hence we look to the Betti numbers.

\begin{prop}\label{prop:N=2 and b=0}
Suppose $r(z,\bar{z})$ has diagonal coefficient matrix with signature pair $(P,2)$.  Suppose $\beta_{1,d}=0$.  Then $P\geq 5$ and $\rho\geq 6$. Thus the SOS Conjecture holds in this case.
\end{prop}
\begin{proof}
Suppose $z^{B^1}$ and $z^{B^2}$ are the two monomial components of the mapping $g$.
We must consider two cases depending on the degree of the syzygy $\sigma(B^1, B^2)$ between them.

The first case is easy.  If the degree of $\sigma(B^1, B^2)$ is at least $d+2$, we claim that we must have at least $6$ distinct generators for $I_f$.  Indeed, because $(I_g)_d \subseteq (I_f)_d$, the latter must contain all the monomials
$$z^{B^j+e_k}, \quad1\leq j \leq 2, \quad 1\leq k \leq 3.$$
Thus each of these monomials must be of the form $z^{A^l+e_m}$ for some monomial $z^{A^l}$ in the generating set for $I_f$.  We claim there must be at least $6$ distinct $A^l$.

Fix $j$.  Suppose there were a single $A^l$ such that
\begin{eqnarray*}
B^j+e_k&=&A^l+e_m\\
B^j+e_{k'}&=&A^l+e_{m'}.
\end{eqnarray*}
Then $e_m-e_k=e_{m'}-e_{k'}$.  If $m=k$ and $m'=k'$, we would have $B^j=A^l$, which is impossible.  Thus $m=m'$ and $k=k'$.  Thus the three distinct monomials $z^{B^j+e_k}$ must also have the form $z^{A^l+e_m}$ for {\it three distinct monomials}  $z^{A^l}$ in $I_f$.  In order for this argument to show that there are $6$ distinct $A^l$, we must show that it is impossible to have  a single $A^l$ with
\begin{equation*}
B^1+e_k = A^l+e_m \quad \text{and} \quad B^2+e_{k'}=A^l+e_{m'}.
\end{equation*}
If this were so, we would have
$$z^{B^1+e_k+e_{m'}} = z^{B^2+e_{k'}+e_m}, $$
i.e., $z_k z_{m'} z^{B^1} = z_{k'} z_m z^{B^2}$,
which implies a syzygy between the generators of $I_g$ of degree at most $d+1$.  This contradiction establishes that in this case $I_f$ has at least $6$ generators.

For the next case, we suppose that $\sigma(B^1,B^2)$ has degree $d+1$. Now it is possible for generators $z^A$ of $I_f$ to generate two distinct monomials in $(I_g)_d$.  As we saw above, such an $A$ satisfies
\begin{equation*}\label{eq:relations between As and Bs}
B^1+e_k = A+e_m \quad \text{and} \quad B^2+e_{k'}=A+e_{m'}.
\end{equation*}
Then  the syzygy between $z^{B^1}$ and $z^{B^2}$ is
\begin{equation*}\label{eq:syzygy between the two negs}
\sigma(B^1,B^2)=z_k z_{m'}\varepsilon(B^1)-z_{k'}z_m\varepsilon(B^2).
\end{equation*}
Note that the sets $\{k,m'\}$ and $\{k',m\}$ must be disjoint.
Now, \eqref{eq:relations between As and Bs} implies that there exist multi-indices $C^1$ and $C^2$ of length $d-2$ such that
\begin{gather}
B^1=C^1+e_m  \quad \text{and} \quad A=C^1+e_k\\
B^2=C^2+e_{m'}  \quad \text{and} \quad A=C^2+e_{k'}.
\end{gather}
Now, because
$$A=C^1+e_k=C^2+e_{k'}, $$
there exists a multi-index $D$ of length $d-3$ such that
\begin{equation*}
C^1=D+e_{k'} \quad \text{and} \quad C^2=D+e_k.
\end{equation*}
We conclude that
\begin{eqnarray}
B^1&=& D + e_{k'}+e_m\nonumber\\
A&=& D + e_k +e_{k'} \label{eq:B1,A,B2}\\
B^2 &=& D+e_k+e_{m'}\nonumber.
\end{eqnarray}

We now understand how $A$, $B^1$ and $B^2$ must be related if $z^A$ generates two different elements of $(I_g)_d$.  We need next to know how many such multi-indices $A$ there are, and how many total elements of $(I_g)_d$ they generate.  We consider cases depending on how the basis vectors $e_k,e_{k'},e_m,e_{m'}$ can be chosen.

In equations \eqref{eq:B1,A,B2}, the sets $\{k',m\}$ and $\{k,m'\}$ must be disjoint, though it is possible that $k'=m$ or $k=m'$.  Now, $k,k',m,m'$ must come from the set $\{1,2,3\}$.  Thus it is in fact necessary that either $k'=m$ or $k=m'$.  We assume without loss of generality that $k'=m$, and furthermore that this index is 1.  Then
$B^1=D+2e_1$. We now have two subcases: $k=m'$ or $k \neq m'$.

\noindent{\it Subcase 1: $k=m'$.} Assume without loss of generality that this index is 2.  Then
$$B^2=D+2e_2 $$
and there is only one possible $A$:
$$A=D+e_1+e_2. $$
Because this is the only $A$ that can give rise to two elements of $(I_g)_d$ and because $(I_g)_d$ contains $6$ distinct monomials, we must have $P \geq 5$ in this subcase.

\noindent{\it Subcase 2: $k \neq m'$.} Thus one of these indices is equal to 2 and the other is equal to 3.
Then
$$B^2=D+e_2+e_3 $$
and there are two possibilities for $A$:
$$A^1=D+e_1+e_2 \quad \text{or} \quad A^2 = D+e_1+e_3. $$
Now, because
$$B^1+e_2=A^1+e_1, \quad B^1+e_3=A^2+e_1, \quad B^2+e_1=A^1+e_3, \quad B^2+e_1=A^2+e_2, $$
at most three elements of $(I_g)_d$ can be generated by $z^{A^1}$ and $z^{A^2}$.  If $I_f$ contains only one of the $z^{A^j}$, we still need another $4$ generators for $I_f$, for a total of at least $5$.  If $I_f$ contains both $z^{A^1}$ and $z^{A^2}$, we still need another $3$, for a total of (again) at least $5$ generators for $I_f$. We have proved the first part of the proposition.

For the second part, note that
\begin{eqnarray*}
\rho&\geq&H_{I_{f\oplus g}}(d) - H_{I_g}(d)\\
&\geq&M(P+2)-6\\
&\geq&M(7)-6=12-6=6.
\end{eqnarray*}
\end{proof}


The next proposition still considers $r$ of signature pair $(P,2)$, but now addresses the case in which $\sigma(B^1,B^2)$ only has degree $d$.  In this case, we get a smaller lower bound on $P$, but $H_{I_g}(d)$ is also smaller.
\begin{prop}\label{prop:P0 for N=2 and b=1.}
Suppose $r(z,\bar{z})$ has diagonal coefficient matrix with signature pair $(P,2)$.  Suppose $\beta_{1,d}=1$. Then $P \geq 4$ and $\rho \geq 5$.  Thus the SOS Conjecture holds in this case.
\end{prop}
\begin{proof}
Because the syzygy between $z^{B^1}$ and $z^{B^2}$ is of degree less than $d+2$, it is possible for a single generator $z^A$ of $I_f$ to give rise to more than one monomial in $(I_g)_d$. For such an $A$, the same argument from the proof of Proposition \ref{prop:N=2 and b=0} still allows us to express $B^1$, $A$, and $B^2$ as in equations \eqref{eq:B1,A,B2}
for some multi-index $D$ of length $d-3$.  The syzygy between $z^{B^1}$ and $z^{B^2}$ can still be expressed by the relationship
$$B^1+e_k+e_{m'} = B^2+e_{k'}+e_m. $$
In this case, however, the syzygy is only of degree $d$. Thus $\{k,m'\}$ and $\{k',m\}$ have precisely one element in common.  We can not have $k=m$, for this would imply $B^1=A$, which is false.  Similarly we can not have $k'=m'$.  We have two remaining cases to analyze.

\noindent{\it Case 1: $k=k'$.} In this case, if we write $C=D+e_k=D+e_{k'}$, we have
\begin{eqnarray*}
B^1&=& C+e_m\\
A&=& C + e_k \\
B^2&=& C+e_{m'}.
\end{eqnarray*}
We are thinking of $m$ and $m'$ as fixed.  There is then only one choice for $A$:
$$A= C+e_k, \; k\neq m,m'.$$
Then
$$B^1 + e_k=A + e_m\quad \text{and} \quad
B^2 + e_k= A+e_{m'}.$$
Thus  the monomial $z^A$ gives rise to $2$ distinct elements of $(I_g)_d$.

\noindent{\it Case 2: $m=m'$.}
Now we just have
$$B^1+e_k=A+e_m=B^2+e_{k'},$$
where $k$ and $k'$ are fixed (determined by the syzygy between $z^{B^1}$ and $z^{B^2}$).  Only one multi-index $A=B^1+e_k-e_m$ satisfies this equation, and it only gives rise to the two (equal) monomials $z^{B^1+e_k}$ and $z^{B^2+e_{k'}}$.  Thus although $z^A$ is needed in $I_f$, it only gives rise to one element of $(I_g)_d$.

Now, because $\beta_{1,d}=1$, $(I_g)_d$ has $5$ distinct monomials. We can use one monomial of the type described in Case 1 to reach two of them.  The remaining three each require one additional monomial in $I_f$, for a total of at least $4$ generators for $I_f$.

Because $P\geq 4$,
\begin{eqnarray*}
\rho&\geq&M(P+2)-5\\
&\geq&M(6)-5=10-5=5.
\end{eqnarray*}
\end{proof}

\section{\label{sec:related ques} A Related Rank Problem}

As we discussed in the introduction, the Sum of Squares Conjecture deals with one possible question we could ask about a bihomogeneous polynomial, namely, What ranks are possible for  $r(z,\bar{z}) \norm{z}^2$ if it is known to be a squared norm? Another question seeks to describe the way the rank of $r(z,\bar{z}) \norm{z}^2$ depends on its bi-degree.

Although one could attempt to answer this question for arbitrary bihomogeneous polynomials, it is a hard question. The question has only been answered completely in the case in which the coefficient matrix of $r$ is diagonal.  It has been answered by Lebl and Peters as part of their consideration of the problem of obtaining sharp degree bounds for proper monomial mappings between balls in complex Euclidean spaces. We will discuss this connection at the end of the paper.  In this section, we discuss their work and give a new proof of one of their theorems using the tools of commutative algebra discussed in Section \ref{sec:alg}.

When $r$ has a diagonal coefficient matrix, the components of the mappings $f$ and $g$ in the holomorphic decomposition of $r$ are monomials.   Thus
\begin{eqnarray*}
r(z,\bar{z})&=&\norm{f(z)}^2 - \norm{g(z)}^2\\
&=&\sum_a |C_a|^2 |z^a|^2-\sum_b |C_b|^2 |z^b|^2,
\end{eqnarray*}
where the first sum is taken over a set of $P$ multi-indices of length $d-1$ and the second sum is taken over a set of $N$ multi-indices of length $d-1$.
If we replace $|z_j|^2$ with $x_j$ and $|C_a|^2$ with $c_a$, we obtain a homogeneous polynomial
$$q(x)=\sum_a c_a x^a, $$
where the sum is taken over all multi-indices of length $d-1$, and precisely $P$ of the $c_a$ are positive and $N$ are negative. With the same change of variables, $\norm{z}^2$ gives the polynomial $s=\sum x_j$, and
the rank of $r(z,\bar{z}) \norm{z}^2$ is the number of distinct monomials appearing in $sq$ with non-zero coefficient.  We again denote this quantity by either $\rho$ or $\rho(sq)$. For the rank problem considered in this section, {\it we do not assume that $r(z,\bar{z})\norm{z}^2$ is a squared norm,} and thus $sq$ may have both positive and negative coefficients.

Previous work on this rank question used a graph-theoretic tool called the Newton diagram of $q$. Although our goal is to replace the Newton diagram of $q$ with
numerical invariants of certain ideals generated by the monomials appearing in
$q$, we briefly discuss the Newton diagram so that we can
describe the relationship between our work and that of D'Angelo,
Kos, and Riehl in \cite{DKR:03} and Lebl and Peters in \cites{LP:11,LP:12}. Let $q(x)=\sum_a c_a x^a$ be a homogeneous polynomial.  Define a graph $\Gamma(q)$ as follows: For the vertex set,  take one vertex for each multi-index $a$ of length $d-1$ for which $c_a$ appears in $q$ with non-zero coefficient. Introduce an edge between the vertices associated with $a$ and $b$ if there exists $j,k \in \{1,\ldots,n\}$ such that $x_j x^a = x_k x^b$.

Lebl and Peters introduce two additional concepts that play an
essential role in their theorem.
\begin{definition}
The Newton diagram of $q$ is {\it connected} if $\Gamma(q)$ is a
connected graph.
\end{definition}
\begin{definition}
$p$ has {\it $\pi$-degree} $\pi(p)$ if $\pi(p)$ is the smallest
integer such that $p(x)=x^{\alpha}\psi(x)$ with $\deg \psi =\pi(p)$.
\end{definition}

Lebl and Peters obtain the following.
\begin{theorem}[Lebl and Peters, \cite{LP:11}] \label{thm:LP for 3 variables}Let $p$ be a homogeneous polynomial in three variables.  Suppose $p=sq$ and that $\Gamma(q)$ is connected. Then
\begin{equation*}\label{eq: p-degree estimate}
\rho(p) \geq \frac{\pi(p)+5}{2}.
\end{equation*}
\end{theorem}

For $q(x) = \sum_a c_a x^a$ a homogeneous polynomial of
degree $d-1$, we are interested in replacing the Newton diagram
with the Betti numbers of associated ideals.  Let $\sA$ be the set
of multi-indices $a$ for which $c_a > 0$, let $\sB$ be the set of
multi-indices for which $c_a < 0$.  Consider ideals $I_f=\langle x^a \rangle_{a
\in \sA}$, $I_g = \langle x^a \rangle_{a \in \sB}$, and $I_{f \oplus g}= \langle x^a \rangle_{a \in \sA \cup \sB}$.  This notation is deliberately chosen to agree with the notation for the ideals formed from the functions appearing in the holomorphic decomposition of $r(z,\bar{z})$, for if the $q$ under consideration was formed through the change of variables described above, these ideals are precisely the same as the ideals considered earlier, except that the variable is now called $x$ instead of $z$.

We will consider the three sets of Betti numbers,
$\{\alpha_{i,j}\}$, $\{\beta_{i,j}\}$, and $\{\gamma_{i,j}\}$
associated with the ideals $I_f$, $I_g$ and $I_{f \oplus g}$, respectively.
\begin{prop}\label{prop:lower bd on rho in terms of betti}
With notation as above,
\begin{equation*}\label{eq: key inequality on betti numbers}
\rho(sq) \geq n\gamma_{0,d-1}-2\gamma_{1,d}+\alpha_{1,d}+\beta_{1,d}.
\end{equation*}
\end{prop}
\begin{proof}
All generators of $I_{f \oplus g}$, $I_f$, and $I_g$ are monomials of degree
$d-1$.  The monomials that could potentially arise in $sq$ are
precisely the monomials in $(I_{f \oplus g})_d$.  The number of such monomials is
$H_{I_{f \oplus g}}(d)$.

Not all potential monomials need actually appear in $sq$ with
non-zero coefficient.  We obtain a lower bound on $\rho(sq)$ if we
count only those monomials in $(I_g)_d$ but not $(I_f)_d$ or in
$(I_f)_d$ but not $(I_g)_d$.  Doing so gives
\begin{eqnarray*}
\rho(sq)&\geq&H_{I_{f \oplus g}}(d) - H_{I_g}(d)    + H_{I_{f \oplus g}}(d) - H_{I_f}(d)\\
&=&2\left[\gamma_{0,d-1} \binom{d-(d-1)+n-1}{n-1} - \gamma_{1,d} \binom{d-d+n-1}{n-1}   \right] \\
& & {} - \left[\alpha_{0,d-1} \binom{d-(d-1)+n-1}{n-1} - \alpha_{1,d} \binom{d-d+n-1}{n-1}   \right]\\
 & &{}- \left[\beta_{0,d-1} \binom{d-(d-1)+n-1}{n-1} - \beta_{1,d} \binom{d-d+n-1}{n-1}   \right]\\
&=&n\gamma_{0,d-1}-2\gamma_{1,d}+\alpha_{1,d}+\beta_{1,d}.
\end{eqnarray*}
\end{proof}

We will use this proposition in the special case in which $n=3$ and $q$ has maximal rank.  In that case, $\gamma_{0,d-1}=\binom{d+1}{2}$ and $H_{I_{f \oplus g}}(d) = \binom{d+2}{2}$, and so $\gamma_{1,d}=d^2-1$.

The idea of the proof of Proposition \ref{prop:lower bd on rho in terms of betti} leads to an important definition:
\begin{definition}
Let
\begin{equation*}
\#(sq)=H_{I_{f \oplus g}}(d)-H_{I_f}(d) + H_{I_{f \oplus g}}(d)-H_{I_g}(d).
\end{equation*}
$\#(sq)$ counts monomials in $(I_{f\oplus g})_d$ not in $(I_f)_d$ and monomials in
$(I_{f \oplus g})_d$ not in $(I_g)_d$. To be consistent with Lebl and Peters, we
call such monomials {\it nodes}.
\end{definition}

In order to use this result to obtain a new proof of Theorem \ref{thm:LP for 3 variables}, we must be able to identify minimal generating sets for the first syzygy module $M_1$ of $I_{f \oplus g}$.  The next lemma is the key.
\begin{lemma}\label{lemma: dependent sets}
Suppose $n=3$ and suppose $I \subset \C[x_1,x_2,x_3]$ is generated by
monomials of degree $d-1$.  Let $A$ be a multi-index of
length $d$ in which no component is zero.  Then
$\sigma(A-e_1,A-e_2)$,
$\sigma(A-e_1,A-e_3)$, and
$\sigma(A-e_2,A-e_3)$ are
dependent, but any pair are independent.
\end{lemma}
\begin{proof}
We have
\begin{eqnarray*}
\sigma(A-e_1,A-e_2)&=&x_1 \varepsilon(A-e_1) - x_2 \varepsilon(A-e_2)\\
\sigma(A-e_1,A-e_3)&=&x_1 \varepsilon(A-e_1) - x_3 \varepsilon(A-e_3)\\
\sigma(A-e_2,A-e_3)&=&x_2\varepsilon(A-e_2) - x_3\varepsilon(A-e_3).
\end{eqnarray*}
Clearly
$\sigma(A-e_1,A-e_3)-\sigma(A-e_1,A-e_2)=\sigma(A-e_2,A-e_3)$.
\end{proof}

The proof of Theorem \ref{thm:LP for 3 variables} involves two steps. We first need a lemma showing that it suffices to prove the result for $q$ full.  Next,
in order to make use of Proposition \ref{prop:lower bd on rho in terms of betti}, we must obtain a good lower bound for $\alpha_{1,d}+\beta_{1,d}$.

\begin{lemma}\label{lemma: filling in}
Suppose that $q$ is a homogeneous polynomial of degree $d-1$ in three variables, that the terms of $q$ have no common monomial factor of positive degree, and that $\Gamma(q)$ is
connected. Then there exists a homogeneous polynomial $q'$ of degree $d-1$with maximal rank , with
$\mathcal{A} \subseteq \mathcal{A}'$ and $\mathcal{B} \subseteq
\mathcal{B}'$, and with
\begin{equation*}
\#(sq) \geq \#(sq').
\end{equation*}
\end{lemma}

Lebl and Peters outline a proof of this lemma. Although it is possible to give a proof that replaces the Newton diagram with more algebraic notation and terminology,  our proof is not fundamentally different and so we omit the details.

We now give a new proof of Theorem \ref{thm:LP for 3 variables}.
\begin{proof}
It suffices to prove the theorem for polynomials $p$ of degree $d$ whose terms have no common monomial factor of positive degree, for if $p$ does not have this property we may replace it with a new polynomial that does have this property and that has degree equal to the $\pi$-degree of $p$.  By Lemma \ref{lemma: filling in}, it also suffices to prove the theorem for $q$ of maximal rank. In light of Proposition \ref{prop:lower bd on rho in terms of betti}, we need only find a sharp lower bound for $\alpha_{1,d}+\beta_{1,d}$.  We drop subscripts for the remainder of the proof.

Because the ideal $I_{f \oplus g}=\langle x^a\rangle_{a \in \mathcal{C}}$ includes all monomials of degree $d-1$, the set consisting of  divided Koszul relations $\sigma(a,b)$ of degree $d$ generates $M_1$, the first syzygy module of $I_{f \oplus g}$.  The set of all such syzygies is dependent; it has only $\gamma_{1,d}=d^2-1$ independent elements. Define two sets of multi-indices {\it of length $d$}:
\begin{eqnarray*}
T&:=&\{A:|A|=d\;\text{and $A_j>0$ for all j}\}\\
E&:=&\{A:|A|=d\;\text{and precisely one component $A_j$ is 0}\}.
\end{eqnarray*}
By Lemma \ref{lemma: dependent sets}, if $A \in T$, only two of the three degree $d$ relations $\sigma(A-e_1,A-e_2)$, $\sigma(A-e_1,A-e_3)$, and $\sigma(A-e_2,A-e_3)$ are independent.  Because $T$ has $\binom{d-1}{2}$ elements, we obtain a set of $2|T|=2\binom{d-1}{2}=d^2-3d+2$ independent syzygies of this type.  Next, take $A \in E$.  Thus $A$ has a zero component, and two nonzero components $A_j$ and $A_k$.  We obtain a single degree $d$ syzygy associated with this $A$, namely $\sigma(A-e_j,A-e_k)$.  Because $E$ has $3(d-1)$ elements, we obtain $|E|=3(d-1)$ degree $d$ syzygies of this type.  Note that,
$$d^2-3d+2-3(d-1)=d^2-1=\gamma_{1,d}, $$
and so we have described a complete generating set for $M_1$.

Let $T^+$ be the subset of $T$ consisting of those $A$ for which $A-e_1, A-e_2,A-e_3 \in \sA$, with a similar definition for $T^-$, and let  $T^0$ be the set of $A$ in $T$ such that at least one of $A-e_1,A-e_2,A-e_3$ is in $\sA$ and at least one is in $\sB$.  We define subsets $E^+$, $E^-$, and $E^0$ of $E$ in a similar manner. The key observation is that these sets give us a way to keep track of how many independent relations $\sigma(a,b)$ involve two elements of $\sA$, how many involve two elements of $\sB$, and how many involve an element of $\sA$ and an element of $\sB$.  Thus if $A \in T^+$, for example, $A-e_j \in \sA$ for $j=1,2,3$, and so all three relations $\sigma(A-e_1,A-e_2)$, $\sigma(A-e_1,A-e_3)$, $\sigma(A-e_2,A-e_3)$ are degree $d$ relations among generators of $I_f$ and are hence relevant in the count $\alpha(=\alpha_{1,d})$.  Because the three relations are dependent but any two are independent, the total contribution to $\alpha$ associated with this element of $T^+$ is two. And thus we see that
\begin{equation*}
\gamma_{1,d}=|E^0|+|E^+| +|E^-| + 2(|T^0| + |T^+| +|T^-|)
\end{equation*}
and
\begin{eqnarray}
\alpha + \beta &=& |E^+| + |E^-| + 2|T^+| +2|T^-| +|T^0|\nonumber\\
&=& \gamma_{1,d}-(|E^0| + |T^0|)\nonumber\\
&=&\gamma_{1,d}-\frac{1}{2}|E^0| -\frac{1}{2}(|E^0|+2|T^0|).\label{eq: expression for alpha + beta}
\end{eqnarray}
Clearly $|E^0| \leq |E|=3(d-1)$.  To obtain an upper bound for $|E^0|+2|T^0|$, we note that this quantity can be interpreted as the total number of independent degree $d$ relations $\sigma(a,b)$ for which one multi-index is in $\sA$ and one is in $\sB$. We count these relations in a different way.  Observe that in every triple of the form $\{b+e_1,b+e_2,b+e_3: |b|=d-2 \}$, at most two of the three relations $\{\sigma(b+e_1,b+e_2),\sigma(b+e_1,b+e_3), \sigma(b+e_2,b+e_3) \}$ can involve one multi-index in $\sA$ and one multi-index in $\sB$.   Because there are precisely $\binom{d}{2}=\frac{1}{2}(d^2-d)$ such triples,
\begin{equation*}
|E^0| + 2|T^0| \leq 2 \binom{d}{2} = d^2 - d.
\end{equation*}
Thus
\begin{eqnarray*}
\rho(sq) &\geq& 3\gamma_{0,d-1}-2\gamma_{1,d}+\alpha+\beta \\
&\geq&3\gamma_{0,d-1}-2\gamma_{1,d}+\gamma_{1,d}-\frac{1}{2}|E^0|-\frac{1}{2}(|E^0|+2|T^0|)\\
&\geq&H_{I_{f \oplus g}}(d)-\frac{3}{2}(d-1)-\frac{1}{2}(d^2 - d)\\
&=&\binom{d+2}{2}-\frac{3}{2}(d-1)-\frac{1}{2}(d^2-d)\\
&=&\frac{d+5}{2}.
\end{eqnarray*}
\end{proof}

\section{\label{sec:finishing proof} Proof of Theorem \ref{thm:SOS for diagonal r in 3 var}}

Suppose $r(z,\bar{z} )$ is a real bihomogeneous polynomial in three variables with bi-degree $(d-1,d-1)$ and diagonal coefficient matrix. Suppose $r(z,\bar{z})\norm{z}^2=\norm{h(z)}^2$.  We must show that the rank $\rho$ of $\norm{h(z)}^2$ is at least $M(k_0+1)=5$.

The results of Section \ref{sec:small numbers of negs} establish the result for such $r$ with signature pair $(P,N)$ for $0 \leq N \leq 2$.  Thus we need only consider  $N\geq 3$.  Because we want to use the results of Section \ref{sec:related ques}, for the remainder of the paper we will consider the equivalent formulation in terms of real homogeneous polynomials discussed in Section \ref{sec:related ques}.  Thus we assume $q$ is a polynomial of degree $d-1$ in three variables, and we prove that $p=sq$ has rank at least 5. The signature pair $(P,N)$ now tells us the numbers of positive and negative coefficients in $q$.  We need only consider $N\geq 3$.

Consider such a $q$.  Its Newton diagram may or may not be connected. If it is not connected, then the Newton diagram is made up of a number of connected components, each corresponding to a homogeneous polynomial $q_j$ of degree $d-1$.  Clearly $q=\sum q_j$.  Also,
$$\rho(sq)=\sum \rho(sq_j).$$
This identity is true because, if $j \neq j'$, no term in $sq_j$ can equal a term in $sq_{j'}$, for if there were such a term, it would imply a syzygy of degree $d$ between a monomial in $q_j$ and a monomial in $q_{j'}$. The existence of such a syzygy would contradict the assumption that the underlying monomials are in different components of $\Gamma(q)$.  In light of this observation, it suffices to prove our assertion about $\rho(sq)$ under the hypothesis that $q$ has a connected Newton diagram. We make this assumption for the remainder of the proof.

Theorem \ref{thm:LP for 3 variables} now gives the result if $\pi(p)\geq 5$, for then $\rho(p) \geq \frac{5+5}{2}=5$.  Thus we need only consider $p$ with $\pi(p) \leq 4$.  Because $p=sq$, all such $p$ correspond to a polynomial $q$ with $\pi(q) \leq 3$.  By dividing out any common factor, we may assume that $q$ itself has degree at most 3.

We treat the few remaining cases for which $q$ has degree at most $3$ and $N$ is at least 3.  Although this part of the argument is tedious, it is not difficult.
Observe, there are $10$ monomials of degree $3$ in $x_1$, $x_2$, and $x_3$.  $I_g$ can not contain any of the monomials $x_j^3$, for if it did, $x_j^4$ would be in $(I_g)_4$ but not in $(I_f)_4$.  Furthermore, $I_g$ can contain at most one of $x_j^2 x_k$ and $x_j x_k^2$, for if it contained both, $x_j^2 x_k^2$ would be an element of $(I_g)_4$ that could not be in $(I_f)_4$.  Let
\begin{eqnarray*}
S_1&=&\{x_1^2 x_2, x_1 x_2^2\}\\
S_2&=&\{x_1^2 x_3, x_1 x_3^2  \}\\
S_3&=&\{x_2^2 x_3, x_2 x_3^2\}.
\end{eqnarray*}
For our first several cases, $I_g$ contains precisely one element of each of the sets $S_1$, $S_2$, $S_3$.  Although our arguments only use a careful enumeration of the monomials in various components of $I_f$ and $I_g$ and the fact that $(I_g)_4 \subseteq (I_f)_4$, we will show the monomials in question on Newton diagrams as a visual aid. See Figure \ref{fig:newton} for the labeling convention in the Newton diagram.

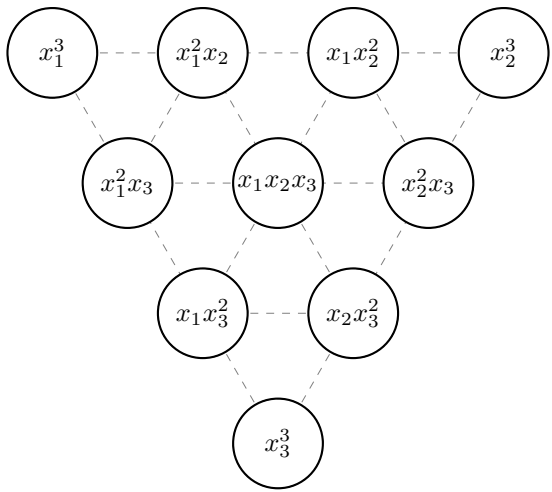
\begin{figure}[h!]
\centering
\begin{tikzpicture}

\begin{scope}[y=(60:2),x=(0:2)]

\draw[help lines, dashed]
 (0,0) -- (0,3)
 (-1,1) -- (-1,3)
 (-2,2) -- (-2,3)
 (0,0) -- (-3,3)
 (0,1) -- (-2,3)
 (0,2) -- (-1,3)
 (-3,3) -- (0,3)
 (-2,2) -- (0,2)
 (-1,1) -- (0,1);

\foreach \y  in {0,...,3}{
\pgfmathtruncatemacro\xmin{int(-1*\y)}
  \foreach \x in
    {\xmin,...,0}{
    \node[draw,circle,minimum size=6,inner sep=12, fill, color=white] at (\x,\y) {};
    \node[draw,circle,minimum size=6,inner sep=12, dashed] at (\x,\y) {};
  }
}

\node at (0,0) {$x_3^3$};
\node[draw,thick, circle,minimum size=6,inner sep=12] at (0,0) {};
\node at (-1,1) {$x_1x_3^2$};
\node[draw,thick, circle,minimum size=6,inner sep=12] at (-1,1) {};
\node at (0,1) {$x_2x_3^2$};
\node[draw,thick, circle,minimum size=6,inner sep=12] at (0,1) {};
\node at (-2,2) {$x_1^2 x_3$};
\node[draw,thick, circle,minimum size=6,inner sep=12] at (-2,2) {};
\node at (-1,2) {$x_1x_2 x_3$};
\node[draw,thick, circle,minimum size=6,inner sep=12] at (-1,2) {};
\node at (0,2) {$x_2^2 x_3$};
\node[draw,thick, circle,minimum size=6,inner sep=12] at (0,2) {};
\node at (0,3) {$x_2^3$};
\node[draw,thick, circle,minimum size=6,inner sep=12] at (0,3) {};
\node at (-1,3) {$x_1 x_2^2$};
\node[draw,thick, circle,minimum size=6,inner sep=12] at (-1,3) {};
\node at (-2,3) {$x_1^2 x_2$};
\node[draw,thick, circle,minimum size=6,inner sep=12] at (-2,3) {};
\node at (-3,3) {$x_1^3$};
\node[draw,thick, circle,minimum size=6,inner sep=12] at (-3,3) {};
\end{scope}
\end{tikzpicture}
\caption{Newton Diagram}
\label{fig:newton}
\end{figure}

\subsection{Case 1.}  Suppose either
\begin{equation}\label{eq:first ideal in case 1}
I_g=\langle x_1^2 x_2, x_1 x_3^2, x_2^2 x_3\rangle
\end{equation}
or
$$I_g=\langle x_1^2 x_3, x_1 x_2^2, x_2 x_3^2 \rangle.$$
These are the same case, for the second ideal is obtained from the first simply by interchanging $x_2$ and $x_3$.  Thus for definiteness, suppose $I_g$ is given by \eqref{eq:first ideal in case 1}. Because $(I_g)_4$ includes the $6$ monomials
$$x_1^3 x_2, x_1^2 x_2^2, x_1^2 x_3^2, x_1 x_3^3, x_2^3 x_3, x_2^2 x_3^2,$$
$I_f$ must include the $6$ monomials
$$x_1^3, x_1^2 x_3, x_1 x_2^2, x_2^3,x_2 x_3^2, x_3^3 .$$
Thus, because $H_{I_g}(d)=9$,
\begin{eqnarray*}
\rho&\geq&M(P+N)-9\\
&\geq&M(9)-9=14-9=5.
\end{eqnarray*}
We illustrate the placement of positive and negative terms on the Newton diagram in Figure \ref{fig:case1}.

\begin{remark}
One can show that the rank is, in fact, at least $6$ by listing elements of $(I_f)_4$ rather than simply using $M(P+N)$ as a lower bound for $H_{I_f}(4)$.  Since we only need to show in each case that the rank is at least 5, we will not trouble ourselves to give the best possible lower bound for $\rho$.
\end{remark}

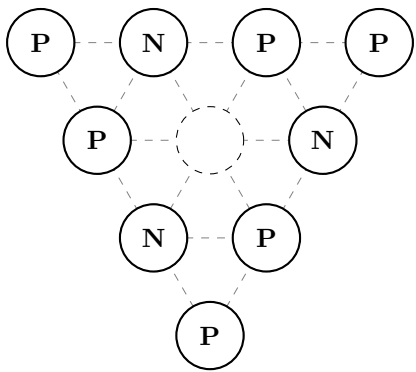
\begin{figure}[h!]
\centering
\begin{tikzpicture}[y=(60:2),x=(0:2),scale=0.75]

\draw[help lines, dashed]
 (0,0) -- (0,3)
 (-1,1) -- (-1,3)
 (-2,2) -- (-2,3)
 (0,0) -- (-3,3)
 (0,1) -- (-2,3)
 (0,2) -- (-1,3)
 (-3,3) -- (0,3)
 (-2,2) -- (0,2)
 (-1,1) -- (0,1);

\foreach \y  in {0,...,3}{
\pgfmathtruncatemacro\xmin{int(-1*\y)}
  \foreach \x in
    {\xmin,...,0}{
    \node[draw,circle,minimum size=6,inner sep=9, fill, color=white] at (\x,\y) {};
    \node[draw,circle,minimum size=6,inner sep=9, dashed] at (\x,\y) {};
  }
}

\node at (0,0) {{\bf P}};
\node[draw,thick, circle,minimum size=6,inner sep=9] at (0,0) {};
\node at (-1,1) {{\bf N}};
\node[draw,thick, circle,minimum size=6,inner sep=9] at (-1,1) {};
\node at (0,1) {{\bf P}};
\node[draw,thick, circle,minimum size=6,inner sep=9] at (0,1) {};
\node at (-2,2) {{\bf P}};
\node[draw,thick, circle,minimum size=6,inner sep=9] at (-2,2) {};
\node at (0,2) {{\bf N}};
\node[draw,thick, circle,minimum size=6,inner sep=9] at (0,2) {};
\node at (0,3) {{\bf P}};
\node[draw,thick, circle,minimum size=6,inner sep=9] at (0,3) {};
\node at (-1,3) {{\bf P}};
\node[draw,thick, circle,minimum size=6,inner sep=9] at (-1,3) {};
\node at (-2,3) {{\bf N}};
\node[draw,thick, circle,minimum size=6,inner sep=9] at (-2,3) {};
\node at (-3,3) {{\bf P}};
\node[draw,thick, circle,minimum size=6,inner sep=9] at (-3,3) {};
\end{tikzpicture}
\caption{Newton Diagram for Case 1.}
\label{fig:case1}
\end{figure}

\subsection{Case 2.} This case is just a slight modification of Case 1.  Add $x_1 x_2 x_3$ to the generating set of either ideal. Although we have changed $N$, we have not changed $(I_g)_4$ because all of the monomials $x_1^2 x_2 x_3$, $x_1 x_2^2 x_3$, $x_1 x_2 x_3^2$ can also be expressed as some $x_j$ multiplied by one of the original three generators.  Thus $H_{I_g}(4) = 9$ again, and by the same argument above, the remaining $6$ monomials of degree $3$ must be in $I_f$.  Thus
\begin{eqnarray*}
\rho&\geq&M(P+N)-9\\
&\geq&M(10)-9=15-9=6.
\end{eqnarray*}

We illustrate the placement of positive and negative terms on the Newton diagram in Figure \ref{fig:case2}.

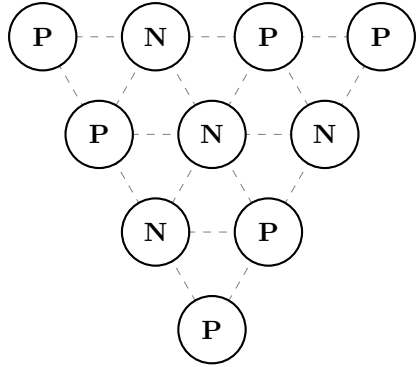
\begin{figure}[h!]
\centering
\begin{tikzpicture}[y=(60:2),x=(0:2),scale=0.75]

\draw[help lines, dashed]
 (0,0) -- (0,3)
 (-1,1) -- (-1,3)
 (-2,2) -- (-2,3)
 (0,0) -- (-3,3)
 (0,1) -- (-2,3)
 (0,2) -- (-1,3)
 (-3,3) -- (0,3)
 (-2,2) -- (0,2)
 (-1,1) -- (0,1);

\foreach \y  in {0,...,3}{
\pgfmathtruncatemacro\xmin{int(-1*\y)}
  \foreach \x in
    {\xmin,...,0}{
    \node[draw,circle,minimum size=6,inner sep=9, fill, color=white] at (\x,\y) {};
    \node[draw,circle,minimum size=6,inner sep=9, dashed] at (\x,\y) {};
  }
}

\node at (0,0) {{\bf P}};
\node[draw,thick, circle,minimum size=6,inner sep=9] at (0,0) {};
\node at (-1,1) {{\bf N}};
\node[draw,thick, circle,minimum size=6,inner sep=9] at (-1,1) {};
\node at (0,1) {{\bf P}};
\node[draw,thick, circle,minimum size=6,inner sep=9] at (0,1) {};
\node at (-2,2) {{\bf P}};
\node[draw,thick, circle,minimum size=6,inner sep=9] at (-2,2) {};
\node at (-1,2) {{\bf N}};
\node[draw,thick, circle,minimum size=6,inner sep=9] at (-1,2) {};
\node at (0,2) {{\bf N}};
\node[draw,thick, circle,minimum size=6,inner sep=9] at (0,2) {};
\node at (0,3) {{\bf P}};
\node[draw,thick, circle,minimum size=6,inner sep=9] at (0,3) {};
\node at (-1,3) {{\bf P}};
\node[draw,thick, circle,minimum size=6,inner sep=9] at (-1,3) {};
\node at (-2,3) {{\bf N}};
\node[draw,thick, circle,minimum size=6,inner sep=9] at (-2,3) {};
\node at (-3,3) {{\bf P}};
\node[draw,thick, circle,minimum size=6,inner sep=9] at (-3,3) {};
\end{tikzpicture}
\caption{Newton Diagram for Case 2.}
\label{fig:case2}
\end{figure}

\subsection{Case 3.} Suppose that either
\begin{equation*}\label{eq:first ideal in case 3}
I_g=\langle x_1^2 x_2, x_1^2 x_3, x_2^2 x_3     \rangle
\end{equation*}
or that $I_g$ is one of the five other ideals obtained by permuting the variables in this generating set. Because $(I_g)_4$ contains
$$x_1^3 x_2, x_1^2 x_2^2, x_1^2 x_3^2, x_2^3 x_3, x_2^2 x_3^2, $$
$I_f$ must contain the $5$ generators
$$x_1^3, x_1 x_2^2, x_1 x_3^2, x_2^3, x_2 x_3^2. $$
Also, because $x_1^2 x_2 x_3$ is in $(I_g)_4$ and we already have $x_1^2 x_2$ and $x_1^2 x_3$ in $I_g$, we need $x_1 x_2 x_3$ in $I_f$.  Thus $P\geq 6$.  Because $H_{I_g}(4)=8$ (because $\beta_{1,d}=1$),
$$\rho \geq M(6+3)-8=14-8=6.$$

We illustrate the placement of positive and negative terms on the Newton diagram in Figure \ref{fig:case3}.

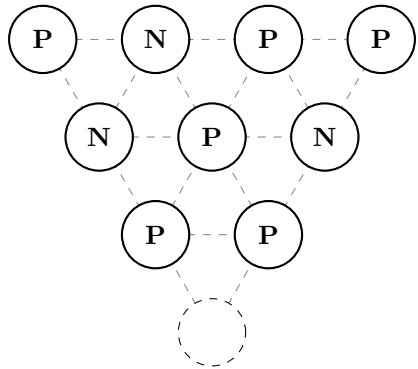
\begin{figure}[h!]
\centering
\begin{tikzpicture}[y=(60:2),x=(0:2),scale=0.75]

\draw[help lines, dashed]
 (0,0) -- (0,3)
 (-1,1) -- (-1,3)
 (-2,2) -- (-2,3)
 (0,0) -- (-3,3)
 (0,1) -- (-2,3)
 (0,2) -- (-1,3)
 (-3,3) -- (0,3)
 (-2,2) -- (0,2)
 (-1,1) -- (0,1);

\foreach \y  in {0,...,3}{
\pgfmathtruncatemacro\xmin{int(-1*\y)}
  \foreach \x in
    {\xmin,...,0}{
    \node[draw,circle,minimum size=6,inner sep=9, fill, color=white] at (\x,\y) {};
    \node[draw,circle,minimum size=6,inner sep=9, dashed] at (\x,\y) {};
  }
}

\node at (-1,1) {{\bf P}};
\node[draw,thick, circle,minimum size=6,inner sep=9] at (-1,1) {};
\node at (0,1) {{\bf P}};
\node[draw,thick, circle,minimum size=6,inner sep=9] at (0,1) {};
\node at (-2,2) {{\bf N}};
\node[draw,thick, circle,minimum size=6,inner sep=9] at (-2,2) {};
\node at (-1,2) {{\bf P}};
\node[draw,thick, circle,minimum size=6,inner sep=9] at (-1,2) {};
\node at (0,2) {{\bf N}};
\node[draw,thick, circle,minimum size=6,inner sep=9] at (0,2) {};
\node at (0,3) {{\bf P}};
\node[draw,thick, circle,minimum size=6,inner sep=9] at (0,3) {};
\node at (-1,3) {{\bf P}};
\node[draw,thick, circle,minimum size=6,inner sep=9] at (-1,3) {};
\node at (-2,3) {{\bf N}};
\node[draw,thick, circle,minimum size=6,inner sep=9] at (-2,3) {};
\node at (-3,3) {{\bf P}};
\node[draw,thick, circle,minimum size=6,inner sep=9] at (-3,3) {};
\end{tikzpicture}
\caption{Newton Diagram for Case 3.}
\label{fig:case3}
\end{figure}

We have now completed our analysis of all those cases in which $I_g$ includes one element from each of the sets $S_1$, $S_2$, $S_3$ above. We next treat the cases in which $I_g$ contains one element from each of two of the sets and does {\it not} contain either of the elements from the third set. {\it Because we are also assuming in all of these remaining cases that $N \geq 3$, in the remaining cases $I_g$ must also contain $x_1 x_2 x_3$.}

\subsection{Case 4.} Suppose that either
$$I_g = \langle x_1^2 x_2, x_1 x_2 x_3, x_1 x_3^2 \rangle $$
or that $I_g$ is one of the other five ideals obtained by permuting the variables in this generating set. Because $(I_g)_4$ contains
$$x_1^3 x_2,x_1^2 x_2^2, x_1^2 x_3^2, x_1 x_3^3    ,$$
$I_f$ must contain the $4$ generators
$$x_1^3, x_1^2 x_3, x_1 x_2^2,  x_3^3 .$$
Furthermore, because $x_1 x_2 x_3$ and $x_1 x_3^2$ are both in $I_g$ and contribute to the monomial $x_1 x_2 x_3^2$ in $(I_g)_4$, we must have $x_2 x_3^2$ in the generating set for $I_f$.  Thus $P\geq 5$. Note also that in this case, $\beta_{1,d}=2$ and so $H_{I_g}(4)=9-2=7$, and
$$\rho \geq M(5+3)-7=13-7=6.$$

We illustrate the placement of positive and negative terms on the Newton diagram in Figure \ref{fig:case4}.

\begin{figure}[h!]
\centering
\begin{tikzpicture}[y=(60:2),x=(0:2),scale=0.75]

\draw[help lines, dashed]
 (0,0) -- (0,3)
 (-1,1) -- (-1,3)
 (-2,2) -- (-2,3)
 (0,0) -- (-3,3)
 (0,1) -- (-2,3)
 (0,2) -- (-1,3)
 (-3,3) -- (0,3)
 (-2,2) -- (0,2)
 (-1,1) -- (0,1);

\foreach \y  in {0,...,3}{
\pgfmathtruncatemacro\xmin{int(-1*\y)}
  \foreach \x in
    {\xmin,...,0}{
    \node[draw,circle,minimum size=6,inner sep=9, fill, color=white] at (\x,\y) {};
    \node[draw,circle,minimum size=6,inner sep=9, dashed] at (\x,\y) {};
  }
}

\node at (0,0) {{\bf P}};
\node[draw,thick, circle,minimum size=6,inner sep=9] at (0,0) {};
\node at (-1,1) {{\bf N}};
\node[draw,thick, circle,minimum size=6,inner sep=9] at (-1,1) {};
\node at (0,1) {{\bf P}};
\node[draw,thick, circle,minimum size=6,inner sep=9] at (0,1) {};
\node at (-2,2) {{\bf P}};
\node[draw,thick, circle,minimum size=6,inner sep=9] at (-2,2) {};
\node at (-1,2) {{\bf N}};
\node[draw,thick, circle,minimum size=6,inner sep=9] at (-1,2) {};
\node at (-1,3) {{\bf P}};
\node[draw,thick, circle,minimum size=6,inner sep=9] at (-1,3) {};
\node at (-2,3) {{\bf N}};
\node[draw,thick, circle,minimum size=6,inner sep=9] at (-2,3) {};
\node at (-3,3) {{\bf P}};
\node[draw,thick, circle,minimum size=6,inner sep=9] at (-3,3) {};
\end{tikzpicture}
\caption{Newton Diagram for Case 4.}
\label{fig:case4}
\end{figure}
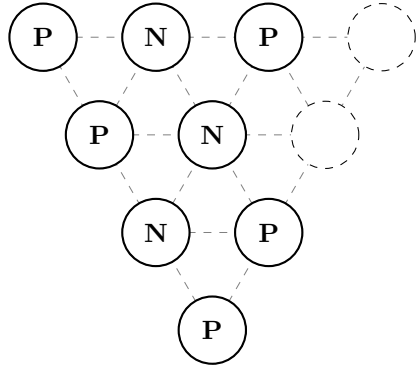

\subsection{Case 5.} Suppose that either
$$I_g = \langle x_1^2 x_2, x_1 x_2 x_3, x_2 x_3^2 \rangle $$
or that $I_g$ is one of the other two ideals obtained by permuting the variables in this generating set. Because $(I_g)_4$ contains
$$x_1^3 x_2, x_1^2 x_2^2, x_2^2 x_3^2, x_2 x_3^3 ,$$
$I_f$ must contain the $4$ monomials
$$x_1^3, x_1 x_2^2, x_2^2 x_3,x_3^3 $$
in its generating set.  Furthermore, because $(I_g)_4$ contains $x_1^2 x_2 x_3$ and $I_g$ contains both $x_1^2 x_2$ and $x_1 x_2 x_3$, the monomial $x_1^2 x_3$ must be in the generating set for $I_f$.  Similarly, because $(I_g)_4$ contains $x_1 x_2 x_3^2$ and $I_g$ contains both $x_1 x_2 x_3$ and $x_2 x_3^2$, the monomial $x_1 x_3^2$ must be in the generating set for $I_f$.  Thus $P\geq 6$, and because $\beta_{1,d}=2$,
$$\rho\geq M(6+3)-7=14-7=7 .$$

We illustrate the placement of positive and negative terms on the Newton diagram in Figure \ref{fig:case5}.

\begin{figure}[h!]
\centering
\begin{tikzpicture}[y=(60:2),x=(0:2),scale=0.75]

\draw[help lines, dashed]
 (0,0) -- (0,3)
 (-1,1) -- (-1,3)
 (-2,2) -- (-2,3)
 (0,0) -- (-3,3)
 (0,1) -- (-2,3)
 (0,2) -- (-1,3)
 (-3,3) -- (0,3)
 (-2,2) -- (0,2)
 (-1,1) -- (0,1);

\foreach \y  in {0,...,3}{
\pgfmathtruncatemacro\xmin{int(-1*\y)}
  \foreach \x in
    {\xmin,...,0}{
    \node[draw,circle,minimum size=6,inner sep=9, fill, color=white] at (\x,\y) {};
    \node[draw,circle,minimum size=6,inner sep=9, dashed] at (\x,\y) {};
  }
}

\node at (0,0) {{\bf P}};
\node[draw,thick, circle,minimum size=6,inner sep=9] at (0,0) {};
\node at (-1,1) {{\bf P}};
\node[draw,thick, circle,minimum size=6,inner sep=9] at (-1,1) {};
\node at (0,1) {{\bf N}};
\node[draw,thick, circle,minimum size=6,inner sep=9] at (0,1) {};
\node at (-2,2) {{\bf P}};
\node[draw,thick, circle,minimum size=6,inner sep=9] at (-2,2) {};
\node at (-1,2) {{\bf N}};
\node[draw,thick, circle,minimum size=6,inner sep=9] at (-1,2) {};
\node at (0,2) {{\bf P}};
\node[draw,thick, circle,minimum size=6,inner sep=9] at (0,2) {};
\node at (-1,3) {{\bf P}};
\node[draw,thick, circle,minimum size=6,inner sep=9] at (-1,3) {};
\node at (-2,3) {{\bf N}};
\node[draw,thick, circle,minimum size=6,inner sep=9] at (-2,3) {};
\node at (-3,3) {{\bf P}};
\node[draw,thick, circle,minimum size=6,inner sep=9] at (-3,3) {};
\end{tikzpicture}
\caption{Newton Diagram for Case 5.}
\label{fig:case5}
\end{figure}
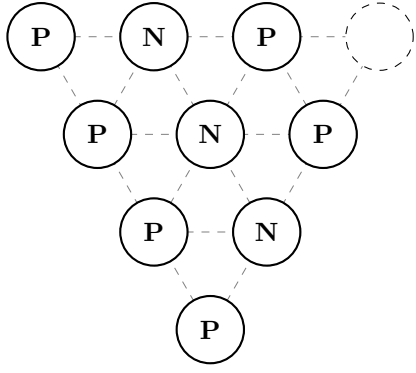

We have now treated all the cases, for it is not possible that $I_g$ contains all of (say) $x_1^2 x_2$, $x_1^2 x_3$, and $x_1 x_2 x_3$, for then $x_1^2 x_2 x_3$ would be in $(I_g)_4$ but not in $(I_f)_4$.  Thus in all cases in which the degree of $q$ is at most $3$, the number of negatives $N\geq 3$, and $(I_g)_4 \subseteq (I_f)_4$, the associated rank is at least $5$.  In light of the comments at the beginning of the section, the proof of Theorem \ref{thm:SOS for diagonal r in 3 var} is now complete.

\section{Sharp Degree Bounds for Proper Monomial Maps from $B_2$ to $B_K$}

We conclude this paper by recalling the connection between Theorem \ref{thm:LP for 3 variables} and the sharp degree estimate theorem for proper monomial mappings between the unit ball in $\mathbb{C}^2$ and the unit ball in $\mathbb{C}^k$.

A mapping $\phi \colon X \to Y$ between topological spaces is {\it proper} if $\phi^{-1}(E)$ is compact in $X$ whenever $E$ is compact in $Y$.  If $B_n$ denotes the unit ball in $\mathbb{C}^n$, if $\phi \colon B_n \to B_k$, and if $\phi$ extends smoothly past the boundary, then $\phi$ is proper if and only if it is non-constant and sends the boundary sphere $S^{2n-1}$ to $S^{2k-1}$, i.e., if
\begin{equation*}
\norm{\phi(z)}^2 = 1 \quad \text{whenever} \quad \norm{z}^2=1.
\end{equation*}
By a theorem of Forstneric \cite{FF:89}, a proper holomorphic map between balls that is sufficiently smooth on the boundary is in fact rational, and thus there is little loss in restricting attention to proper rational maps. It is a long-term goal in the field of CR geometry to classify all proper rational mappings from $B_n$ to $B_k$.

An interesting finding is that the complexity of such proper mappings is related to the domain and target dimensions $n$ and $k$. For example, Faran  \cite{JF:86} showed that for $n \leq k\leq 2n-2$, every proper rational map from $B_n$ to $B_k$ is of degree one, whereas if $k \leq 3$, every proper rational map from $B_2$ to $B_k$ is of degree at most three \cite{JF:82}. These are some of the earliest {\it degree estimate} results, i.e., theorems that bound the degree $d$ of a proper rational mapping from $B_n$ to $B_k$ in terms of some function of $n$ and $k$.  When all components of the mapping are monomials, the sharp degree estimate theorem is known:
\begin{theorem} [D'Angelo, Kos, and Riehl \cite{DKR:03}; Lebl and Peters \cite{LP:11,LP:12}]
\label{thm:deg est}
Let $\phi \colon B_n \to B_k$ be a proper monomial mapping of degree $d$. Then
\begin{equation}\label{eq: degre estimates}
d \leq \begin{cases}2k-3 & n=2\\
\frac{k-1}{n-1}  & n > 2 ,\end{cases}
\end{equation}
and the inequalities are sharp.
\end{theorem}
D'Angelo, Kos, and Riehl \cite{DKR:03} prove the inequality for $n=2$ and show that it is sharp by exhibiting a family of mappings $\phi_d \colon B_2 \to B_k$ of odd degree for which $d=2k-3$.  Lebl and Peters prove the inequality for $n>2$. In the monomial case, this theorem follows from Theorem \ref{thm:LP for 3 variables} and its analogue for larger numbers of variables.  The purpose of this section is to briefly indicate the connection.  Lebl and Peters are responsible for making this connection.

Let $\phi \colon B_n \to B_k$ be a proper polynomial mapping.
If each component of $\phi$ is a monomial $C_a z^a$, then
we require
$$\norm{\phi(z)}^2=\sum |C_a|^2  \prod_{j=1}^n |z_j|^{2a_j} \quad\text{whenever} \quad \sum_{j=1}^n |z_j|^2=1.$$
Replace $(|z_1|^2,\ldots,|z_n|^2)$ with $(x_1,\ldots,x_n)$.  We see that there is a correspondence between proper monomial mappings and polynomials $\tilde{p}$ in $x=(x_1,\ldots,x_n)$ with non-negative coefficients such that
\begin{equation}\label{eq: def of class P}
\tilde{p}(x)=\sum c_a x^a = 1 \quad\text{whenever}\quad \sum_{j=1}^n x_j=1.
\end{equation}
We let $\sP$ denote the class of all polynomials in $\C[x_1,\ldots,x_n]$ satisfying \eqref{eq: def of class P}. It has been shown by D'Angelo, Kos, and Riehl \cite{DKR:03} for $n=2$ and by Lebl and Peters \cite{LP:11,LP:12} for $n>2$ that the degree estimates \eqref{eq: degre estimates} hold for elements of $\sP$, where $k=\rho(\tilde{p})$ is the number of distinct monomials appearing in $\tilde{p}$ with non-zero coefficient.  Furthermore, the estimates are sharp.

We prefer to deal with homogeneous polynomials.  Thus in
\eqref{eq: def of class P}, we homogenize with $x_{n+1}$ and then
replace $x_{n+1}$ with $-x_{n+1}$ to obtain
\begin{equation*}
p(x_1,\ldots,x_{n+1})=0 \quad \text{whenever}\quad \sum_{j=1}^{n+1} x_j =
0
\end{equation*}
where $p(x_1,\ldots,x_n, -1)=\tilde{p}(x_1,\ldots,x_n)-1$.  In other words,
if $\tilde{p} \in \sP$, the polynomial $p$ obtained by homogenizing
$\tilde{p}-1$ in this way must satisfy precisely the hypothesis on signs
appearing in the next theorem.  Thus
Theorem \ref{thm:deg est} and Theorem \ref{thm: deg est homog form} are
equivalent.

\begin{theorem}\label{thm: deg est homog form}
Let $p$ be homogeneous of degree $d$ with
$p(0, \ldots, 0, x_{n+1})=(-x_{n+1})^d$.  Suppose (i) the polynomial $\tilde{p}$
defined by $p(x_1,\ldots, x_n,-1)=\tilde{p}(x_1,\ldots,x_n)-1$ has all
non-negative coefficients and (ii) $p=sq$. Then
\begin{equation*}
\rho(p) \geq \begin{cases}\frac{d+5}{2} & n=2,\\
d(n-1)+2 & n>2.\end{cases}
\end{equation*}
and the inequalities are sharp.
\end{theorem}

The hypotheses in Theorem \ref{thm: deg est homog form} are awkward, so Lebl and Peters sought weaker hypotheses under which the lower bound on $\rho(p)$ still holds. In three variables, their theorem is Theorem \ref{thm:LP for 3 variables} above.
Thus our new, more algebraic proof of Theorem \ref{thm:LP for 3 variables} constitutes a new proof of the sharp degree bound theorem for proper monomials mappings from $B_2$ to $B_k$.

\bibliographystyle{amsalpha}
\bibliography{references}

\end{document}